\documentclass{amsart}

\usepackage{graphicx}

\newtheorem{theorem}{Theorem}[section]

\theoremstyle{definition}
\newtheorem{definition}[theorem]{Definition}

\theoremstyle{remark}

\numberwithin{equation}{section}

\newcommand{\abs}[1]{\lvert#1\rvert}

\newcommand{\blankbox}[2]{%
  \parbox{\columnwidth}{\centering
    \setlength{\fboxsep}{0pt}%
    \fbox{\raisebox{0pt}[#2]{\hspace{#1}}}%
  }%
}

\usepackage{tikz}
\usetikzlibrary{arrows}
\usepackage{caption}

\usepackage{multicol}

\newenvironment{Figure}
  {\par\medskip\noindent\minipage{\linewidth}}
  {\endminipage\par\medskip}

\begin{document}

\title{Framework of two-dimensional functional walks}

\author{Fabian Schneider}
\address{Max-Planck Institute for Mathematics, Bonn, Germany}
\email{fschn@mpim-bonn.mpg.de}

\date{September 2017}

\begin{abstract}
This paper gives a general introduction to two-dimensional functional walks with particular attention to notation and definition. We also give applications of functional walks and a visual overview of some walks generated by $f(n)=n^2$ and $f(n)=n^3$.
\end{abstract}

\maketitle

\section{Definition and Notation}

Let $v_n$ denote graph vertices given as vectors from the origin and let $e_n = v_n - v_{n-1}$ denote the graph edges as vectors from $v_{n-1}$ to $v_n$. A walk $W$ is defined as sequence $v_0, e_1, v_1, e_2, v_2, ...$ \cite{A}:

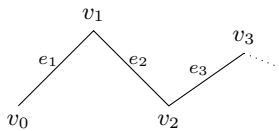
\begin{figure}[!h]
\centering
\begin{tikzpicture}
\draw (0,0) node[align=left,below] {\small $v_0$} -- node[pos=0.4,align=left,above] {\scriptsize $e_1$}
(1,1) node[align=left,above] {\small $v_1$} -- node[pos=0.6,align=left,above] {\scriptsize $e_2$}
(2,0) node[align=left,below] {\small $v_2$} -- node[pos=0.4,align=left,above] {\scriptsize $e_3$}
(3,0.7) node[align=left,above] {\small $v_3$};
\draw[dotted] (3,0.7) -- (3.5,0.5);
\end{tikzpicture}
\end{figure}

\noindent A functional walk $W$ on a two-dimensional plane is iteratively generated using the functions $f : \mathbb{N} \rightarrow \mathbb{N}$ and $g : \mathbb{N} \rightarrow \mathbb{R}$ and the whole integer $m$. The function $f(n)$ determines the direction $\varphi(n)$ of the edge $e_n$ by dividing $2\pi$ into $m$ parts, i.e. an unit circle into $m$ possible directions:
\begin{align*}
    \varphi(n) = \dfrac{2\pi}{m} \cdot ( f(n) ~\text{mod}~ m ).
\end{align*}
The length of the edge $e_n$ is given by $g(n)$. We define that $v_0 = (0,0)$. Thus the vertex vector $v_n$ is given by the sum of edge vectors
\begin{align*}
    v_n = v_0 + e_1 + e_2 + ... + e_{n} = \sum_{t=1}^n g(t)\binom{\cos \varphi(t)}{\sin \varphi(t)}.
\end{align*}
Since a two-dimensional functional walk is defined by $f$, $g$ and $m$, we will introduce the notation $W:\langle f ~|~ g\rangle_m$. In case of $g(n)=1$, we write $W:\langle f \rangle_m$ and omit $g(n)$. Note that the set of vertices $v_n$ and edges $e_n$ to describe a walk is in general given as infinite in context of functional walks. We will use the terms \emph{functional walk} and \emph{walk} interchangeable. The concept of functional walks can be categorized to dynamical systems \cite{ds} and iterated function systems \cite{ifs}.


\section{Closed and Open walks}

\begin{definition}
    A walk $W$ with vertices $v_i$ and edges $e_i$ is called \emph{closed} if there exists a finite set $\{v_k,e_k,v_{k+1},e_{k+1},...,v_{K},e_{K} \}$ so that all $v_i, e_i$ are include in this set for any value of $i$. Otherwise the walk is called \emph{open}.
\end{definition}
The edges and vertices of the walk $W_S:\langle n \rangle_4$ repeat after four steps and end at the origin, thus the walk describes a square and we call it closed. The Walk $W_{Sp} : \langle n | n \rangle_4$ describes a square spiral which vertices and edges do not repeat, thus we call it open.

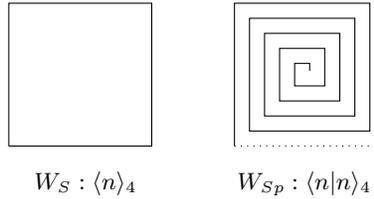
\begin{figure}[!h]
\centering
\begin{tikzpicture}
\draw (0,0) -- (0,1.9) -- (1.9,1.9) -- (1.9,0) -- (0,0);
\node at (1.02,-0.5) {\small $W_S:\langle n \rangle_4$};

\draw (4.0,1.0) -- (4.0,1.0) -- (4.0,1.1) -- (3.8,1.1) -- (3.8,0.8) -- (4.2,0.8) -- (4.2,1.3) -- (3.6,1.3) -- (3.6,0.6) -- (4.4,0.6) -- (4.4,1.5) -- (3.4,1.5) -- (3.4,0.4) -- (4.6,0.4) -- (4.6,1.7) -- (3.2,1.7) -- (3.2,0.2) -- (4.8,0.2) -- (4.8,1.9) -- (3.0,1.9) -- (3.0,0.0);
\draw[dotted] (3.0,0.0) -- (5.0,0.0);
\node at (3.93,-0.5) {\small $W_{Sp} : \langle n | n \rangle_4$};
\end{tikzpicture}
\end{figure}

\noindent Note that all illustrations of open walks posses dots at the \textit{end} to distinguish them from closed walks. 

\section{Bounded and Unbounded walks}

Let $|v_n|$ denote the distance of the $n$th vertex from the origin. There exist walks such as $W_{Sp} : \langle n | n \rangle_4$ that continuously increase in size, i.e. the distance of the vertices from the origin becomes arbitrarily large. Other walks such as $W_B : \langle n | 1/n \rangle_4$, which describes a square spiral that turns from the outside to the inside, are limited in size.

\begin{definition}
    If there exists a constant $B$ such that $B \geq |v_n|$ for any vertex $v_n$ of the walk $W$, we call $W$ \emph{bounded}. Otherwise we call $W$ an \emph{unbounded} walk.
\end{definition}

\noindent Note that closed walks are always bounded with $B = \max\{|v_0|,~ |v_1|,~ |v_2|,~ ...\}$.

\section{Converging walks}

The vertices of the walk $W: \langle n | 1/n \rangle_4$ will come arbitrary close to the point $(0,0)$. Therefore we say that this graph converges.

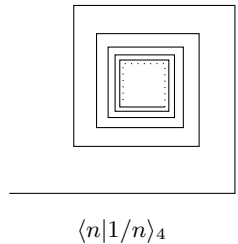
\begin{figure}[!h]
\centering
\begin{tikzpicture}
\draw (0.0,0.0) -- (3.0,0.0) -- (3.0,2.5) -- (0.8571,2.5) -- (0.8571,0.625) -- (2.5238,0.625) -- (2.5238,2.125) -- (1.1602,2.125) -- (1.1602,0.875) -- (2.314,0.875) -- (2.314,1.9464) -- (1.314,1.9464) -- (1.314,1.0089) -- (2.1964,1.0089) -- (2.1964,1.8423) -- (1.4069,1.8423) -- (1.4069,1.0923) -- (2.1212,1.0923) -- (2.1212,1.7741) -- (1.469,1.7741) -- (1.469,1.1491) -- (2.069,1.1491) ;
\draw[dotted] (2.069,1.1491) -- (2.069,1.726) -- (1.5135,1.726) -- (1.5135,1.1903);
\node at (1.5,-0.5) {\small $\langle n | 1/n \rangle_4$};
\end{tikzpicture}
\end{figure}

\begin{definition}
    If there exists a constant $C$ such that $|C - |v_n|| < \varepsilon$ for any arbitrary small positive number $\varepsilon$, we call $W$ a converging walk.
\end{definition}

\section{Repetition Walks}

Let $W_O$ denote some open walk. Although the distance of the vertices of $W_O$ from the origin increase and you can not reduce the walk to a closed figure, it may consist out of repetitions of a single finite walk. 

\begin{definition}
    We call a walk \emph{repetitive} if there exists a \emph{repetition index} $R$ so that for any whole integer $k$
    \begin{align*}
        0 = \sum^{R}_{i=1} | (v_{Rk+i} - v_{Rk}) - v_i |,
    \end{align*}
    i.e. cutting the walk consecutively in pieces that contain $R$ vertices and moving all pieces to the origin $v_0$ results in a single finite walk with vertices $v_0, v_1,...,v_R$ which we call the \emph{repetition figure}.
\end{definition}

Consider the walk $\langle n^2 \rangle_5$ as shown in the figure below. This repetitive walk has a repetition index of $R = 5$. 

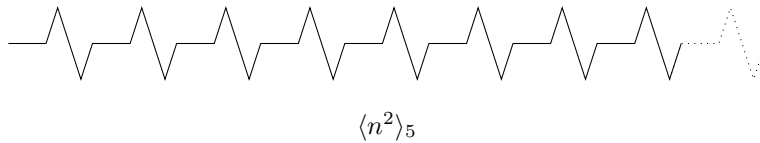
\begin{figure}[!ht]
\centering
\begin{tikzpicture}
\draw (1.118,-0.0) -- (1.618,-0.0) -- (1.7725,0.4755) -- (1.9271,-0.0) -- (2.0816,-0.4755) -- (2.2361,-0.0) -- (2.7361,-0.0) -- (2.8906,0.4755) -- (3.0451,-0.0) -- (3.1996,-0.4755) -- (3.3541,-0.0) -- (3.8541,-0.0) -- (4.0086,0.4755) -- (4.1631,-0.0) -- (4.3176,-0.4755) -- (4.4721,-0.0) -- (4.9721,-0.0) -- (5.1266,0.4755) -- (5.2812,-0.0) -- (5.4357,-0.4755) -- (5.5902,-0.0) -- (6.0902,-0.0) -- (6.2447,0.4755) -- (6.3992,-0.0) -- (6.5537,-0.4755) -- (6.7082,-0.0) -- (7.2082,-0.0) -- (7.3627,0.4755) -- (7.5172,-0.0) -- (7.6717,-0.4755) -- (7.8262,-0.0) -- (8.3262,-0.0) -- (8.4807,0.4755) -- (8.6353,-0.0) -- (8.7898,-0.4755) -- (8.9443,-0.0) -- (9.4443,-0.0) -- (9.5988,0.4755) -- (9.7533,-0.0) -- (9.9078,-0.4755) -- (10.0623,-0.0);
\draw[dotted] (9.9078,-0.4755) -- (10.0623,-0.0) -- (10.5623,-0.0) -- (10.7168,0.4755) -- (10.8713,-0.0) -- (11.0258,-0.4755) -- (11.1803,-0.0);
\end{tikzpicture}
\caption*{$\langle n^{2} \rangle_{5}$}
\end{figure}

The repetition figure is shown as dotted shape at the end. Note that all illustrations of repetition walks posses a dotted shape of the repetition figure.

\section{Evolving walks}

Given the functions $f$ and $g$ of the walk $W$ it is the constant $m$ that determines the shape of the walk. The shape of some walks evolves from an initial walk for small $m$ to a more detailed shape for increasing $m$. This can be observed for various walks such as $\langle n^2 \rangle_{4m+2}$. With increasing $m$ we see how the initial simple structure becomes a detailed Cornu spiral \cite{cornu}.

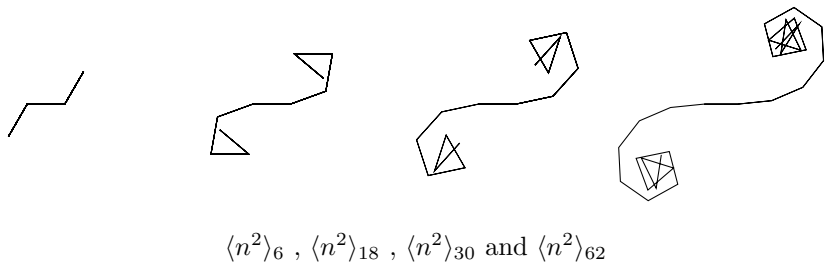
\begin{figure}[!ht]
\centering
\begin{tikzpicture}
\draw (0.0,0.0) -- (0.5,0.0) -- (0.75,0.433) -- (0.5,0.0) -- (-0.0,0.0) -- (-0.25,-0.433) -- (-0.0,0.0) -- (0.5,0.0) -- (0.75,0.433) -- (0.5,0.0) -- (-0.0,0.0) -- (-0.25,-0.433) -- (-0.0,0.0) -- (0.5,0.0) -- (0.75,0.433) -- (0.5,0.0) -- (-0.0,0.0) -- (-0.25,-0.433) -- (-0.0,0.0) -- (0.5,0.0) -- (0.75,0.433) -- (0.5,0.0) -- (-0.0,0.0) -- (-0.25,-0.433) -- (-0.0,0.0) -- (0.5,0.0) -- (0.75,0.433) -- (0.5,0.0) -- (-0.0,0.0) -- (-0.25,-0.433) -- (-0.0,0.0) -- (0.5,0.0) -- (0.75,0.433) -- (0.5,0.0) -- (-0.0,0.0) -- (-0.25,-0.433) -- (-0.0,0.0) -- (0.5,0.0) -- (0.75,0.433) -- (0.5,0.0) -- (-0.0,0.0) -- (-0.25,-0.433) -- (-0.0,0.0) -- (0.5,0.0) -- (0.75,0.433) -- (0.5,0.0);

\draw (3.0,0.0) -- (3.5,0.0) -- (3.9698,0.171) -- (4.0567,0.6634) -- (3.5567,0.6634) -- (3.9397,0.342) -- (3.5567,0.6634) -- (4.0567,0.6634) -- (3.9698,0.171) -- (3.5,0.0) -- (3.0,0.0) -- (2.5302,-0.171) -- (2.4433,-0.6634) -- (2.9433,-0.6634) -- (2.5603,-0.342) -- (2.9433,-0.6634) -- (2.4433,-0.6634) -- (2.5302,-0.171) -- (3.0,0.0) -- (3.5,0.0) -- (3.9698,0.171) -- (4.0567,0.6634) -- (3.5567,0.6634) -- (3.9397,0.342) -- (3.5567,0.6634) -- (4.0567,0.6634) -- (3.9698,0.171) -- (3.5,0.0) -- (3.0,0.0) -- (2.5302,-0.171) -- (2.4433,-0.6634) -- (2.9433,-0.6634) -- (2.5603,-0.342) -- (2.9433,-0.6634) -- (2.4433,-0.6634) -- (2.5302,-0.171) -- (3.0,0.0) -- (3.5,0.0) -- (3.9698,0.171) -- (4.0567,0.6634) -- (3.5567,0.6634) -- (3.9397,0.342) -- (3.5567,0.6634) -- (4.0567,0.6634) -- (3.9698,0.171) -- (3.5,0.0);

\draw (6.0,0.0) -- (6.5,0.0) -- (6.9891,0.104) -- (7.3236,0.4755) -- (7.1691,0.9511) -- (6.6801,0.8471) -- (6.9301,0.4141) -- (7.0846,0.8896) -- (6.75,0.518) -- (7.0846,0.8896) -- (6.9301,0.4141) -- (6.6801,0.8471) -- (7.1691,0.9511) -- (7.3236,0.4755) -- (6.9891,0.104) -- (6.5,0.0) -- (6.0,0.0) -- (5.5109,-0.104) -- (5.1764,-0.4755) -- (5.3309,-0.9511) -- (5.8199,-0.8471) -- (5.5699,-0.4141) -- (5.4154,-0.8896) -- (5.75,-0.518) -- (5.4154,-0.8896) -- (5.5699,-0.4141) -- (5.8199,-0.8471) -- (5.3309,-0.9511) -- (5.1764,-0.4755) -- (5.5109,-0.104) -- (6.0,0.0) -- (6.5,0.0) -- (6.9891,0.104) -- (7.3236,0.4755) -- (7.1691,0.9511) -- (6.6801,0.8471) -- (6.9301,0.4141) -- (7.0846,0.8896) -- (6.75,0.518) -- (7.0846,0.8896) -- (6.9301,0.4141) -- (6.6801,0.8471) -- (7.1691,0.9511) -- (7.3236,0.4755) -- (6.9891,0.104) -- (6.5,0.0);

\draw (9.0,0.0) -- (9.45,0.0) -- (9.8977,0.0455) -- (10.3112,0.223) -- (10.5867,0.5788) -- (10.5639,1.0283) -- (10.1945,1.2853) -- (9.8011,1.0669) -- (9.9139,0.6313) -- (10.3547,0.7219) -- (10.1984,1.1439) -- (9.8569,0.8508) -- (10.2863,0.716) -- (10.0881,1.1201) -- (10.02,0.6753) -- (10.258,1.0571) -- (9.948,0.731) -- (10.258,1.0571) -- (10.02,0.6753) -- (10.0881,1.1201) -- (10.2863,0.716) -- (9.8569,0.8508) -- (10.1984,1.1439) -- (10.3547,0.7219) -- (9.9139,0.6313) -- (9.8011,1.0669) -- (10.1945,1.2853) -- (10.5639,1.0283) -- (10.5867,0.5788) -- (10.3112,0.223) -- (9.8977,0.0455) -- (9.45,-0.0) -- (9.0,0.0) -- (8.5523,-0.0455) -- (8.1388,-0.223) -- (7.8633,-0.5788) -- (7.8861,-1.0283) -- (8.2555,-1.2853) -- (8.6489,-1.0669) -- (8.5361,-0.6313) -- (8.0953,-0.7219) -- (8.2516,-1.1439) -- (8.5931,-0.8508) -- (8.1637,-0.716) -- (8.3619,-1.1201) -- (8.43,-0.6753);

\end{tikzpicture}
\caption*{$\langle n^{2} \rangle_{6}$ , $\langle n^{2} \rangle_{18}$ , $\langle n^{2} \rangle_{30}$ and $\langle n^{2} \rangle_{62}$}
\end{figure}

Note that all these walks are closed and that the size increases due to constant $g=1$ but an increasing number of edges that make up the figure.

\section{Equalities of Walks}

In order to compare two walks we introduce following definition:

\begin{definition}
    Let $v_1,...$ denote the vertices of $\langle f_1 | g_1 \rangle_{m_1}$ and $w_1,...$ the vertices of $\langle f_2 | g_2 \rangle_{m_2}$. We write $\langle f_1 | g_1 \rangle_{m_1} = \langle f_2 | g_2 \rangle_{m_2}$ if $|v_i - w_i| = 0$ for all $i$.
\end{definition}

However, some walks have the exact same shape but different sizes due to different $g$ such as $\langle n | n \rangle_{4}$ and $\langle n | n+5 \rangle_{4}$. In order to express that two walks have the same shape we introduce another definition:

\begin{definition}
    We write $\langle f_1 | g_1 \rangle_{m_1} \equiv \langle f_2 | g_2 \rangle_{m_2}$ if there exists a real $j$ so that $\langle f_1 | g_1 \rangle_{m_1} = \langle f_2 | g_2 + j \rangle_{m_2}$ holds true.
\end{definition}

\section{Regular Polygons}

\begin{theorem}
    A $m$-sided regular polygon is described by the walk $W_p : \langle n \rangle_m$.
\end{theorem}
\begin{proof}
    Since $m$ is a constant value and $f(n)=n$ increases linear, the edges point consecutively in the 1th, 2th, ..., $m$th direction, i.e. the angle between two edges differs always by $\varphi(n+1)-\varphi(n) = 2\pi / m$. Thus the $m$th edge connects back to the origin and because all the edges have the same length we can observe a $m$-sided regular polygon.
\end{proof}

The figure below shows the first walks generated by $\langle n \rangle_m$. Note that the size of the figures was reduced in favor of the alignment.

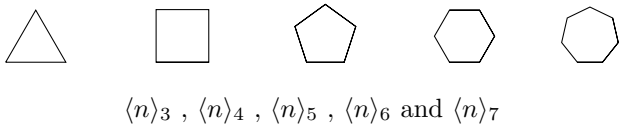
\begin{figure}[!ht]
\centering
\begin{tikzpicture}

\draw (0.0,0.0) -- (0.8,0.0) -- (0.4,0.6928) -- (-0.0,0.0);

\draw (2.0001,0.0) -- (2.7001,0.0) -- (2.7,0.7001) -- (2.0,0.7001) -- (2.0,0.0) -- (2.7001,0.0);

\draw (4.0,0.0) -- (4.5,0.0) -- (4.6545,0.4755) -- (4.25,0.7694) -- (3.8455,0.4755) -- (4.0,0.0) -- (4.5,0.0) -- (4.6545,0.4755) -- (4.25,0.7694) -- (3.8455,0.4755) -- (4.0,0.0);

\draw (5.9,0.0) -- (6.3,0.0) -- (6.5,0.3464) -- (6.3,0.6928) -- (5.9,0.6928) -- (5.7,0.3464) -- (5.9,-0.0) -- (6.3,-0.0) -- (6.5,0.3464) -- (6.3,0.6928) -- (5.9,0.6928);

\draw (7.6,0.0) -- (7.935,0.0) -- (8.1439,0.2619) -- (8.0693,0.5885) -- (7.7675,0.7339) -- (7.4657,0.5885) -- (7.3911,0.2619) -- (7.6,-0.0) -- (7.935,-0.0) -- (8.1439,0.2619) -- (8.0693,0.5885);

\end{tikzpicture}
\caption*{$\langle n \rangle_{3}$ , $\langle n \rangle_{4}$ , $\langle n \rangle_{5}$ , $\langle n \rangle_{6}$ and $\langle n \rangle_{7}$}
\end{figure}

In addition, the interpretation of a circle as polygon with infinite sides allows a circle to be given by the walk $W_c : \left\langle n \right\rangle_m$ with $m \longrightarrow \infty$.

\section{Applications and Conclusion}

Functional walks appear as an elegant method to express geometric figures and objects. The walk notation allows a compact representation of complex shapes with simple functions and it enables to express non-smooth objects such as polygons and more (see appendix) with smooth functions e.g. polynomials. The underlying mathematical principles generate a variety of sophisticated and detailed figures which could be used by graphic algorithms and more.  \\

The concepts of functional walks may also be used for e.g. geometric proofs of theorems. However, there are a variety of aspects of functional walks in general and the shape of certain walks that require further understanding and study.

\newpage
\appendix

\section{Collection of Walks with $f(n)=n^2$}

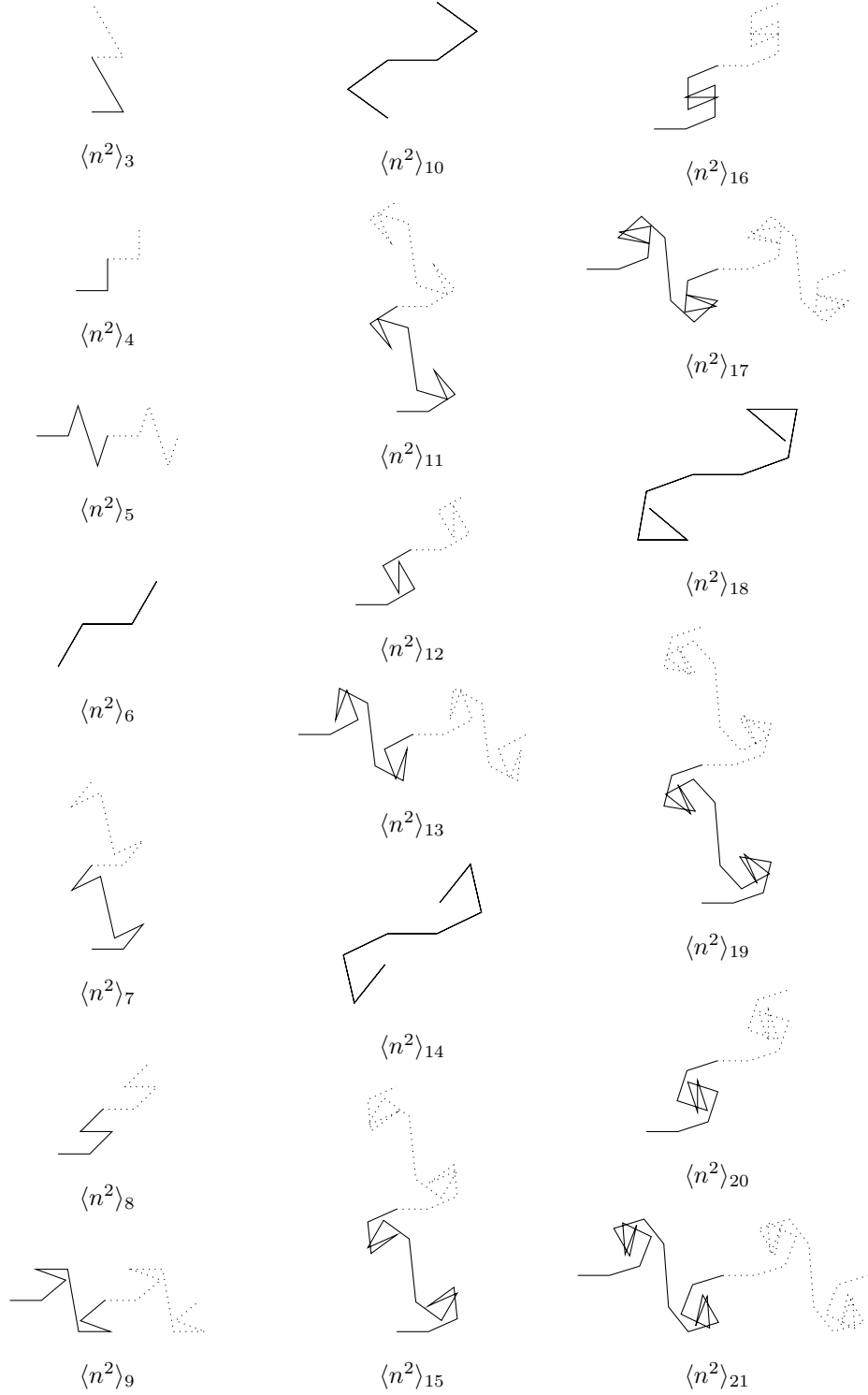
\begin{multicols}{3}

\begin{Figure}
\centering
\begin{tikzpicture}
\draw (0.0,0.0) -- (0.45,0.0) -- (0.225,0.3897) -- (0.0,0.7794);
\draw[dotted] (0.0,0.7794) -- (0.45,0.7794) -- (0.225,1.1691) -- (0.0,1.5588);
\end{tikzpicture}
\captionof*{ }{$\langle n^2 \rangle_{3}$}
\end{Figure}
\begin{Figure}
\centering
\begin{tikzpicture}
\draw (0.0,0.0) -- (0.45,0.0) -- (0.45,0.45);
\draw[dotted] (0.45,0.45) -- (0.9,0.45) -- (0.9,0.9);
\end{tikzpicture}
\captionof*{ }{$\langle n^2 \rangle_{4}$}
\end{Figure}
\begin{Figure}
\centering
\begin{tikzpicture}
\draw (0.0,0.0) -- (0.45,0.0) -- (0.5891,0.428) -- (0.7281,0.0) -- (0.8672,-0.428) -- (1.0062,0.0);
\draw[dotted] (1.0062,0.0) -- (1.4562,0.0) -- (1.5953,0.428) -- (1.7343,0.0) -- (1.8734,-0.428) -- (2.0125,0.0);
\end{tikzpicture}
\captionof*{ }{$\langle n^2 \rangle_{5}$}
\end{Figure}
\begin{Figure}
\centering
\begin{tikzpicture}
\draw (0.0,0.0) -- (0.7,0.0) -- (1.05,0.6062) -- (0.7,0.0) -- (0.0,0.0) -- (-0.35,-0.6062) -- (0.0,0.0);
\end{tikzpicture}
\captionof*{ }{$\langle n^2 \rangle_{6}$}
\end{Figure}
\begin{Figure}
\centering
\begin{tikzpicture}
\draw (0.0,0.0) -- (0.45,0.0) -- (0.7306,0.3518) -- (0.3251,0.1566) -- (0.225,0.5953) -- (0.1249,1.034) -- (-0.2806,0.8388) -- (0.0,1.1906);
\draw[dotted] (0.0,1.1906) -- (0.45,1.1906) -- (0.7306,1.5424) -- (0.3251,1.3472) -- (0.225,1.7859) -- (0.1249,2.2246) -- (-0.2806,2.0294) -- (0.0,2.3812);
\end{tikzpicture}
\captionof*{ }{$\langle n^2 \rangle_{7}$}
\end{Figure}
\begin{Figure}
\centering
\begin{tikzpicture}
\draw (0.0,0.0) -- (0.45,0.0) -- (0.7682,0.3182) -- (0.3182,0.3182) -- (0.6364,0.6364);
\draw[dotted] (0.6364,0.6364) -- (1.0864,0.6364) -- (1.4046,0.9546) -- (0.9546,0.9546) -- (1.2728,1.2728);
\end{tikzpicture}
\captionof*{ }{$\langle n^2 \rangle_{8}$}
\end{Figure}
\begin{Figure}
\centering
\begin{tikzpicture}
\draw (0.0,0.0) -- (0.45,0.0) -- (0.7947,0.2893) -- (0.3719,0.4432) -- (0.8219,0.4432) -- (0.9,0.0) -- (0.9781,-0.4432) -- (1.4281,-0.4432) -- (1.0053,-0.2893) -- (1.35,0.0);
\draw[dotted] (1.35,0.0) -- (1.8,0.0) -- (2.1447,0.2893) -- (1.7219,0.4432) -- (2.1719,0.4432) -- (2.25,0.0) -- (2.3281,-0.4432) -- (2.7781,-0.4432) -- (2.3553,-0.2893) -- (2.7,0.0);
\end{tikzpicture}
\captionof*{ }{$\langle n^2 \rangle_{9}$}
\end{Figure}
\begin{Figure}
\centering
\begin{tikzpicture}
\draw (0.0,0.0) -- (0.7,0.0) -- (1.2663,0.4114) -- (0.7,0.8229) -- (1.2663,0.4114) -- (0.7,0.0) -- (0.0,0.0) -- (-0.5663,-0.4114) -- (0.0,-0.8229) -- (-0.5663,-0.4114) -- (0.0,0.0);
\end{tikzpicture}
\captionof*{ }{$\langle n^2 \rangle_{10}$}
\end{Figure}
\begin{Figure}
\centering
\begin{tikzpicture}
\draw (0.0,0.0) -- (0.45,0.0) -- (0.8286,0.2433) -- (0.5339,0.5834) -- (0.7208,0.174) -- (0.289,0.3008) -- (0.225,0.7462) -- (0.161,1.1917) -- (-0.2708,1.3184) -- (-0.0839,0.9091) -- (-0.3786,1.2492) -- (0.0,1.4925);
\draw[dotted] (0.0,1.4925) -- (0.45,1.4925) -- (0.8286,1.7358) -- (0.5339,2.0759) -- (0.7208,1.6665) -- (0.289,1.7933) -- (0.225,2.2387) -- (0.161,2.6841) -- (-0.2708,2.8109) -- (-0.0839,2.4016) -- (-0.3786,2.7417) -- (0.0,2.985);
\end{tikzpicture}
\captionof*{ }{$\langle n^2 \rangle_{11}$}
\end{Figure}
\begin{Figure}
\centering
\begin{tikzpicture}
\draw (0.0,0.0) -- (0.45,0.0) -- (0.8397,0.225) -- (0.6147,0.6147) -- (0.6147,0.1647) -- (0.3897,0.5544) -- (0.7794,0.7794);
\draw[dotted] (0.7794,0.7794) -- (1.2294,0.7794) -- (1.6191,1.0044) -- (1.3941,1.3941) -- (1.3941,0.9441) -- (1.1691,1.3338) -- (1.5588,1.5588);
\end{tikzpicture}
\captionof*{ }{$\langle n^2 \rangle_{12}$}
\end{Figure}
\begin{Figure}
\centering
\begin{tikzpicture}
\draw (0.0,0.0) -- (0.45,0.0) -- (0.8485,0.2091) -- (0.6889,0.6299) -- (0.5293,0.2091) -- (0.5836,0.6558) -- (0.982,0.4467) -- (1.0362,0.0) -- (1.0905,-0.4467) -- (1.4889,-0.6558) -- (1.5432,-0.2091) -- (1.3836,-0.6299) -- (1.224,-0.2091) -- (1.6225,0.0);
\draw[dotted] (1.6225,0.0) -- (2.0725,0.0) -- (2.471,0.2091) -- (2.3114,0.6299) -- (2.1518,0.2091) -- (2.206,0.6558) -- (2.6045,0.4467) -- (2.6587,0.0) -- (2.713,-0.4467) -- (3.1114,-0.6558) -- (3.1657,-0.2091) -- (3.0061,-0.6299) -- (2.8465,-0.2091) -- (3.245,0.0);
\end{tikzpicture}
\captionof*{ }{$\langle n^2 \rangle_{13}$}
\end{Figure}
\begin{Figure}
\centering
\begin{tikzpicture}
\draw (0.0,0.0) -- (0.7,0.0) -- (1.3307,0.3037) -- (1.1749,0.9862) -- (0.7385,0.4389) -- (1.1749,0.9862) -- (1.3307,0.3037) -- (0.7,0.0) -- (0.0,0.0) -- (-0.6307,-0.3037) -- (-0.4749,-0.9862) -- (-0.0385,-0.4389) -- (-0.4749,-0.9862) -- (-0.6307,-0.3037) -- (0.0,0.0);
\end{tikzpicture}
\captionof*{ }{$\langle n^2 \rangle_{14}$}
\end{Figure}
\begin{Figure}
\centering
\begin{tikzpicture}
\draw (0.0,0.0) -- (0.45,0.0) -- (0.8611,0.183) -- (0.8141,0.6306) -- (0.45,0.3661) -- (0.8611,0.5491) -- (0.6361,0.1594) -- (0.272,0.4239) -- (0.225,0.8714) -- (0.178,1.319) -- (-0.1861,1.5835) -- (-0.4111,1.1937) -- (0.0,1.3768) -- (-0.3641,1.1123) -- (-0.4111,1.5598) -- (0.0,1.7428);
\draw[dotted] (0.0,1.7428) -- (0.45,1.7428) -- (0.8611,1.9259) -- (0.8141,2.3734) -- (0.45,2.1089) -- (0.8611,2.2919) -- (0.6361,1.9022) -- (0.272,2.1667) -- (0.225,2.6143) -- (0.178,3.0618) -- (-0.1861,3.3263) -- (-0.4111,2.9366) -- (0.0,3.1196) -- (-0.3641,2.8551) -- (-0.4111,3.3027) -- (0.0,3.4857);
\end{tikzpicture}
\captionof*{ }{$\langle n^2 \rangle_{15}$}
\end{Figure}
\begin{Figure}
\centering
\begin{tikzpicture}
\draw (0.0,0.0) -- (0.45,0.0) -- (0.8657,0.1722) -- (0.8657,0.6222) -- (0.45,0.45) -- (0.9,0.45) -- (0.4843,0.2778) -- (0.4843,0.7278) -- (0.9,0.9);
\draw[dotted] (0.9,0.9) -- (1.35,0.9) -- (1.7657,1.0722) -- (1.7657,1.5222) -- (1.35,1.35) -- (1.8,1.35) -- (1.3843,1.1778) -- (1.3843,1.6278) -- (1.8,1.8);
\end{tikzpicture}
\captionof*{ }{$\langle n^2 \rangle_{16}$}
\end{Figure}
\begin{Figure}
\centering
\begin{tikzpicture}
\draw (0.0,0.0) -- (0.45,0.0) -- (0.8696,0.1626) -- (0.9111,0.6106) -- (0.4688,0.528) -- (0.8884,0.3654) -- (0.4461,0.4481) -- (0.7786,0.7512) -- (1.1112,0.4481) -- (1.1527,0.0) -- (1.1942,-0.4481) -- (1.5268,-0.7512) -- (1.8593,-0.4481) -- (1.417,-0.3654) -- (1.8366,-0.528) -- (1.3943,-0.6106) -- (1.4358,-0.1626) -- (1.8554,0.0);
\draw[dotted] (1.8554,0.0) -- (2.3054,0.0) -- (2.725,0.1626) -- (2.7665,0.6106) -- (2.3242,0.528) -- (2.7438,0.3654) -- (2.3015,0.4481) -- (2.634,0.7512) -- (2.9666,0.4481) -- (3.0081,0.0) -- (3.0496,-0.4481) -- (3.3822,-0.7512) -- (3.7147,-0.4481) -- (3.2724,-0.3654) -- (3.692,-0.528) -- (3.2497,-0.6106) -- (3.2912,-0.1626) -- (3.7108,0.0);
\end{tikzpicture}
\captionof*{ }{$\langle n^2 \rangle_{17}$}
\end{Figure}
\begin{Figure}
\centering
\begin{tikzpicture}
\draw (0.0,0.0) -- (0.7,0.0) -- (1.3578,0.2394) -- (1.4793,0.9288) -- (0.7793,0.9288) -- (1.3156,0.4788) -- (0.7793,0.9288) -- (1.4793,0.9288) -- (1.3578,0.2394) -- (0.7,0.0) -- (0.0,0.0) -- (-0.6578,-0.2394) -- (-0.7793,-0.9288) -- (-0.0793,-0.9288) -- (-0.6156,-0.4788) -- (-0.0793,-0.9288) -- (-0.7793,-0.9288) -- (-0.6578,-0.2394) -- (0.0,0.0);
\end{tikzpicture}
\captionof*{ }{$\langle n^2 \rangle_{18}$}
\end{Figure}
\begin{Figure}
\centering
\begin{tikzpicture}
\draw (0.0,0.0) -- (0.45,0.0) -- (0.8756,0.1461) -- (0.9861,0.5823) -- (0.5422,0.6564) -- (0.7883,0.2797) -- (0.6076,0.6918) -- (0.9627,0.4154) -- (0.5669,0.2012) -- (0.2622,0.5323) -- (0.225,0.9808) -- (0.1878,1.4292) -- (-0.1169,1.7603) -- (-0.5127,1.5461) -- (-0.1576,1.2697) -- (-0.3383,1.6818) -- (-0.0922,1.3051) -- (-0.5361,1.3792) -- (-0.4256,1.8154) -- (0.0,1.9615);
\draw[dotted] (0.0,1.9615) -- (0.45,1.9615) -- (0.8756,2.1076) -- (0.9861,2.5438) -- (0.5422,2.6179) -- (0.7883,2.2412) -- (0.6076,2.6533) -- (0.9627,2.3769) -- (0.5669,2.1627) -- (0.2622,2.4938) -- (0.225,2.9423) -- (0.1878,3.3907) -- (-0.1169,3.7218) -- (-0.5127,3.5076) -- (-0.1576,3.2312) -- (-0.3383,3.6433) -- (-0.0922,3.2666) -- (-0.5361,3.3407) -- (-0.4256,3.7769) -- (0.0,3.923);
\end{tikzpicture}
\captionof*{ }{$\langle n^2 \rangle_{19}$}
\end{Figure}
\begin{Figure}
\centering
\begin{tikzpicture}
\draw (0.0,0.0) -- (0.45,0.0) -- (0.878,0.1391) -- (1.017,0.567) -- (0.5891,0.7061) -- (0.7281,0.2781) -- (0.7281,0.7281) -- (0.8672,0.3001) -- (0.4392,0.4392) -- (0.5783,0.8672) -- (1.0062,1.0062);
\draw[dotted] (1.0062,1.0062) -- (1.4562,1.0062) -- (1.8842,1.1453) -- (2.0233,1.5733) -- (1.5953,1.7123) -- (1.7343,1.2843) -- (1.7343,1.7343) -- (1.8734,1.3064) -- (1.4454,1.4454) -- (1.5845,1.8734) -- (2.0125,2.0125);
\end{tikzpicture}
\captionof*{ }{$\langle n^2 \rangle_{20}$}
\end{Figure}
\begin{Figure}
\centering
\begin{tikzpicture}
\draw (0.0,0.0) -- (0.45,0.0) -- (0.88,0.1326) -- (1.0444,0.5515) -- (0.639,0.7468) -- (0.6726,0.298) -- (0.837,0.7169) -- (0.7369,0.2782) -- (0.5119,0.6679) -- (0.9419,0.8006) -- (1.2225,0.4487) -- (1.2561,0.0) -- (1.2897,-0.4487) -- (1.5703,-0.8006) -- (2.0003,-0.6679) -- (1.7753,-0.2782) -- (1.6752,-0.7169) -- (1.8396,-0.298) -- (1.8732,-0.7468) -- (1.4677,-0.5515) -- (1.6322,-0.1326) -- (2.0622,0.0);
\draw[dotted] (2.0622,0.0) -- (2.5122,0.0) -- (2.9422,0.1326) -- (3.1066,0.5515) -- (2.7011,0.7468) -- (2.7348,0.298) -- (2.8992,0.7169) -- (2.799,0.2782) -- (2.574,0.6679) -- (3.004,0.8006) -- (3.2846,0.4487) -- (3.3182,0.0) -- (3.3519,-0.4487) -- (3.6324,-0.8006) -- (4.0624,-0.6679) -- (3.8374,-0.2782) -- (3.7373,-0.7169) -- (3.9017,-0.298) -- (3.9353,-0.7468) -- (3.5299,-0.5515) -- (3.6943,-0.1326) -- (4.1243,0.0);
\end{tikzpicture}
\captionof*{ }{$\langle n^2 \rangle_{21}$}
\end{Figure}

\end{multicols}

\section{Collection of Walks with $f(n)=n^3$}

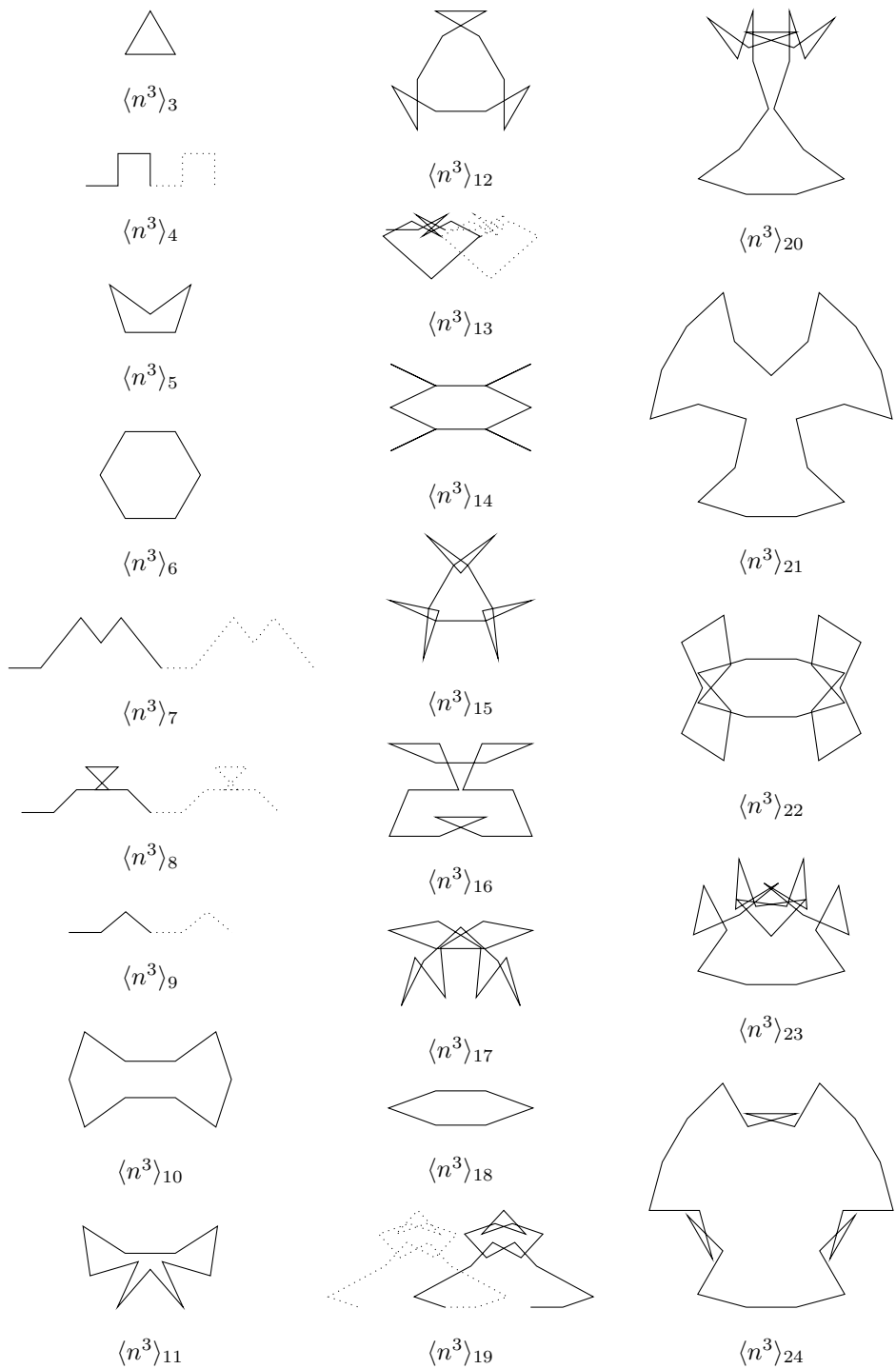
\begin{multicols}{3}

\begin{Figure}
\centering
\begin{tikzpicture}
\draw (0.0,0.0) -- (0.7,0.0) -- (0.35,0.6062) -- (0.0,0.0);
\end{tikzpicture}
\captionof*{ }{$\langle n^3 \rangle_{3}$}
\end{Figure}
\begin{Figure}
\centering
\begin{tikzpicture}
\draw (0.0,0.0) -- (0.45,0.0) -- (0.45,0.45) -- (0.9,0.45) -- (0.9,0.0);
\draw[dotted] (0.9,0.0) -- (1.35,0.0) -- (1.35,0.45) -- (1.8,0.45) -- (1.8,0.0);
\end{tikzpicture}
\captionof*{ }{$\langle n^3 \rangle_{4}$}
\end{Figure}
\begin{Figure}
\centering
\begin{tikzpicture}
\draw (0.0,0.0) -- (0.7,0.0) -- (0.9163,0.6657) -- (0.35,0.2543) -- (-0.2163,0.6657) -- (0.0,0.0);
\end{tikzpicture}
\captionof*{ }{$\langle n^3 \rangle_{5}$}
\end{Figure}
\begin{Figure}
\centering
\begin{tikzpicture}
\draw (0.0,0.0) -- (0.7,0.0) -- (1.05,0.6062) -- (0.7,1.2124) -- (0.0,1.2124) -- (-0.35,0.6062) -- (0.0,0.0);
\end{tikzpicture}
\captionof*{ }{$\langle n^3 \rangle_{6}$}
\end{Figure}
\begin{Figure}
\centering
\begin{tikzpicture}
\draw (0.0,0.0) -- (0.45,0.0) -- (0.7306,0.3518) -- (1.0111,0.7036) -- (1.2917,0.3518) -- (1.5723,0.7036) -- (1.8529,0.3518) -- (2.1334,0.0);
\draw[dotted] (2.1334,0.0) -- (2.5834,0.0) -- (2.864,0.3518) -- (3.1446,0.7036) -- (3.4251,0.3518) -- (3.7057,0.7036) -- (3.9863,0.3518) -- (4.2668,0.0);
\end{tikzpicture}
\captionof*{ }{$\langle n^3 \rangle_{7}$}
\end{Figure}
\begin{Figure}
\centering
\begin{tikzpicture}
\draw (0.0,0.0) -- (0.45,0.0) -- (0.7682,0.3182) -- (1.2182,0.3182) -- (0.9,0.6364) -- (1.35,0.6364) -- (1.0318,0.3182) -- (1.4818,0.3182) -- (1.8,0.0);
\draw[dotted] (1.8,0.0) -- (2.25,0.0) -- (2.5682,0.3182) -- (3.0182,0.3182) -- (2.7,0.6364) -- (3.15,0.6364) -- (2.8318,0.3182) -- (3.2818,0.3182) -- (3.6,0.0);
\end{tikzpicture}
\captionof*{ }{$\langle n^3 \rangle_{8}$}
\end{Figure}
\begin{Figure}
\centering
\begin{tikzpicture}
\draw (0.0,0.0) -- (0.45,0.0) -- (0.7947,0.2893) -- (1.1394,0.0);
\draw[dotted] (1.1394,0.0) -- (1.5894,0.0) -- (1.9342,0.2893) -- (2.2789,0.0);
\end{tikzpicture}
\captionof*{ }{$\langle n^3 \rangle_{9}$}
\end{Figure}
\begin{Figure}
\centering
\begin{tikzpicture}
\draw (0.0,0.0) -- (0.7,0.0) -- (1.2663,0.4114) -- (1.4826,-0.2543) -- (1.2663,-0.92) -- (0.7,-0.5086) -- (0.0,-0.5086) -- (-0.5663,-0.92) -- (-0.7826,-0.2543) -- (-0.5663,0.4114) -- (0.0,0.0);
\end{tikzpicture}
\captionof*{ }{$\langle n^3 \rangle_{10}$}
\end{Figure}
\begin{Figure}
\centering
\begin{tikzpicture}
\draw (0.0,0.0) -- (0.7,0.0) -- (1.2889,0.3784) -- (1.1893,-0.3144) -- (0.5176,-0.1172) -- (0.8084,-0.754) -- (0.35,-0.2249) -- (-0.1084,-0.754) -- (0.1824,-0.1172) -- (-0.4893,-0.3144) -- (-0.5889,0.3784) -- (0.0,0.0);
\end{tikzpicture}
\captionof*{ }{$\langle n^3 \rangle_{11}$}
\end{Figure}
\begin{Figure}
\centering
\begin{tikzpicture}
\draw (0.0,0.0) -- (0.7,0.0) -- (1.3062,0.35) -- (0.9562,-0.2562) -- (0.9562,0.4438) -- (0.6062,1.05) -- (0.0,1.4) -- (0.7,1.4) -- (0.0938,1.05) -- (-0.2562,0.4438) -- (-0.2562,-0.2562) -- (-0.6062,0.35) -- (0.0,0.0);
\end{tikzpicture}
\captionof*{ }{$\langle n^3 \rangle_{12}$}
\end{Figure}
\begin{Figure}
\centering
\begin{tikzpicture}
\draw (0.0,0.0) -- (0.45,0.0) -- (0.8485,0.2091) -- (0.5116,-0.0893) -- (0.9101,0.1198) -- (1.3085,-0.0893) -- (0.9717,-0.3877) -- (0.6349,-0.6861) -- (0.298,-0.3877) -- (-0.0388,-0.0893) -- (0.3597,0.1198) -- (0.7581,-0.0893) -- (0.4213,0.2091) -- (0.8198,0.0);
\draw[dotted] (0.8198,0.0) -- (1.2698,0.0) -- (1.6682,0.2091) -- (1.3314,-0.0893) -- (1.7298,0.1198) -- (2.1283,-0.0893) -- (1.7915,-0.3877) -- (1.4546,-0.6861) -- (1.1178,-0.3877) -- (0.781,-0.0893) -- (1.1794,0.1198) -- (1.5779,-0.0893) -- (1.241,0.2091) -- (1.6395,0.0);
\end{tikzpicture}
\captionof*{ }{$\langle n^3 \rangle_{13}$}
\end{Figure}
\begin{Figure}
\centering
\begin{tikzpicture}
\draw (0.0,0.0) -- (0.7,0.0) -- (1.3307,0.3037) -- (0.7,0.0) -- (1.3307,-0.3037) -- (0.7,-0.6074) -- (1.3307,-0.9112) -- (0.7,-0.6074) -- (0.0,-0.6074) -- (-0.6307,-0.9112) -- (0.0,-0.6074) -- (-0.6307,-0.3037) -- (0.0,0.0) -- (-0.6307,0.3037) -- (0.0,0.0);
\end{tikzpicture}
\captionof*{ }{$\langle n^3 \rangle_{14}$}
\end{Figure}
\begin{Figure}
\centering
\begin{tikzpicture}
\draw (0.0,0.0) -- (0.7,0.0) -- (1.3395,0.2847) -- (0.6548,0.1392) -- (0.8711,-0.5266) -- (0.7979,0.1696) -- (0.4479,0.7758) -- (-0.1184,1.1873) -- (0.35,0.6671) -- (0.8184,1.1873) -- (0.2521,0.7758) -- (-0.0979,0.1696) -- (-0.1711,-0.5266) -- (0.0452,0.1392) -- (-0.6395,0.2847) -- (0.0,0.0);
\end{tikzpicture}
\captionof*{ }{$\langle n^3 \rangle_{15}$}
\end{Figure}
\begin{Figure}
\centering
\begin{tikzpicture}
\draw (0.0,0.0) -- (0.7,0.0) -- (1.3467,0.2679) -- (0.6467,0.2679) -- (0.3788,-0.3788) -- (1.0788,-0.3788) -- (1.3467,-1.0256) -- (0.6467,-1.0256) -- (0.0,-0.7577) -- (0.7,-0.7577) -- (0.0533,-1.0256) -- (-0.6467,-1.0256) -- (-0.3788,-0.3788) -- (0.3212,-0.3788) -- (0.0533,0.2679) -- (-0.6467,0.2679) -- (0.0,0.0);
\end{tikzpicture}
\captionof*{ }{$\langle n^3 \rangle_{16}$}
\end{Figure}
\begin{Figure}
\centering
\begin{tikzpicture}
\draw (0.0,0.0) -- (0.7,0.0) -- (1.3527,0.2529) -- (0.6646,0.3815) -- (0.0695,0.013) -- (0.1341,-0.684) -- (-0.2878,-0.1254) -- (-0.4793,-0.7987) -- (-0.1673,-0.1721) -- (0.35,0.2995) -- (0.8673,-0.1721) -- (1.1793,-0.7987) -- (0.9878,-0.1254) -- (0.5659,-0.684) -- (0.6305,0.013) -- (0.0354,0.3815) -- (-0.6527,0.2529) -- (0.0,0.0);
\end{tikzpicture}
\captionof*{ }{$\langle n^3 \rangle_{17}$}
\end{Figure}
\begin{Figure}
\centering
\begin{tikzpicture}
\draw (0.0,0.0) -- (0.7,0.0) -- (1.3578,0.2394) -- (0.7,0.4788) -- (0.0,0.4788) -- (-0.6578,0.2394) -- (0.0,0.0);
\end{tikzpicture}
\captionof*{ }{$\langle n^3 \rangle_{18}$}
\end{Figure}
\begin{Figure}
\centering
\begin{tikzpicture}
\draw (0.0,0.0) -- (0.45,0.0) -- (0.8756,0.1461) -- (0.4799,0.3603) -- (0.0841,0.5745) -- (-0.2207,0.9055) -- (-0.6164,0.6914) -- (-0.9212,1.0224) -- (-0.4956,1.1686) -- (-0.07,1.0224) -- (-0.3748,1.3535) -- (-0.6795,1.0224) -- (-0.2539,1.1686) -- (0.1717,1.0224) -- (-0.1331,0.6914) -- (-0.5288,0.9055) -- (-0.8336,0.5745) -- (-1.2294,0.3603) -- (-1.6252,0.1461) -- (-1.1995,0.0);
\draw[dotted] (-1.1995,0.0) -- (-0.7495,0.0) -- (-0.3239,0.1461) -- (-0.7197,0.3603) -- (-1.1154,0.5745) -- (-1.4202,0.9055) -- (-1.816,0.6914) -- (-2.1208,1.0224) -- (-1.6951,1.1686) -- (-1.2695,1.0224) -- (-1.5743,1.3535) -- (-1.8791,1.0224) -- (-1.4535,1.1686) -- (-1.0278,1.0224) -- (-1.3326,0.6914) -- (-1.7284,0.9055) -- (-2.0332,0.5745) -- (-2.4289,0.3603) -- (-2.8247,0.1461) -- (-2.3991,0.0);
\end{tikzpicture}
\captionof*{ }{$\langle n^3 \rangle_{19}$}
\end{Figure}
\begin{Figure}
\centering
\begin{tikzpicture}
\draw (0.0,0.0) -- (0.7,0.0) -- (1.3657,0.2163) -- (0.7994,0.6278) -- (0.388,1.1941) -- (0.6043,1.8598) -- (0.6043,2.5598) -- (0.8206,1.8941) -- (1.2321,2.4604) -- (0.6657,2.0489) -- (0.0,2.2652) -- (0.7,2.2652) -- (0.0343,2.0489) -- (-0.5321,2.4604) -- (-0.1206,1.8941) -- (0.0957,2.5598) -- (0.0957,1.8598) -- (0.312,1.1941) -- (-0.0994,0.6278) -- (-0.6657,0.2163) -- (0.0,0.0);
\end{tikzpicture}
\captionof*{ }{$\langle n^3 \rangle_{20}$}
\end{Figure}
\begin{Figure}
\centering
\begin{tikzpicture}
\draw (0.0,0.0) -- (0.7,0.0) -- (1.3689,0.2063) -- (0.8558,0.6825) -- (0.7,1.3649) -- (1.3689,1.5712) -- (2.0378,1.3649) -- (1.882,2.0473) -- (1.532,2.6536) -- (1.0189,3.1297) -- (0.8631,2.4472) -- (0.35,1.9711) -- (-0.1631,2.4472) -- (-0.3189,3.1297) -- (-0.832,2.6536) -- (-1.182,2.0473) -- (-1.3378,1.3649) -- (-0.6689,1.5712) -- (0.0,1.3649) -- (-0.1558,0.6825) -- (-0.6689,0.2063) -- (0.0,0.0);
\end{tikzpicture}
\captionof*{ }{$\langle n^3 \rangle_{21}$}
\end{Figure}
\begin{Figure}
\centering
\begin{tikzpicture}
\draw (0.0,0.0) -- (0.7,0.0) -- (1.3716,0.1972) -- (0.9132,0.7262) -- (1.0129,1.4191) -- (1.6017,1.0407) -- (1.311,0.4039) -- (1.6017,-0.2328) -- (1.0129,-0.6113) -- (0.9132,0.0816) -- (1.3716,0.6106) -- (0.7,0.8078) -- (0.0,0.8078) -- (-0.6716,0.6106) -- (-0.2132,0.0816) -- (-0.3129,-0.6113) -- (-0.9017,-0.2328) -- (-0.6109,0.4039) -- (-0.9017,1.0407) -- (-0.3129,1.4191) -- (-0.2132,0.7262) -- (-0.6716,0.1972) -- (0.0,0.0);
\end{tikzpicture}
\captionof*{ }{$\langle n^3 \rangle_{22}$}
\end{Figure}
\begin{Figure}
\centering
\begin{tikzpicture}
\draw (0.0,0.0) -- (0.7,0.0) -- (1.374,0.1889) -- (0.9704,0.7607) -- (1.2924,1.3823) -- (1.4348,0.6969) -- (0.7928,0.9758) -- (0.2498,1.4175) -- (0.8479,1.0538) -- (0.8001,1.7522) -- (0.5657,1.0926) -- (-0.1278,1.1879) -- (0.35,0.6763) -- (0.8278,1.1879) -- (0.1343,1.0926) -- (-0.1001,1.7522) -- (-0.1479,1.0538) -- (0.4502,1.4175) -- (-0.0928,0.9758) -- (-0.7348,0.6969) -- (-0.5924,1.3823) -- (-0.2704,0.7607) -- (-0.674,0.1889) -- (0.0,0.0);
\end{tikzpicture}
\captionof*{ }{$\langle n^3 \rangle_{23}$}
\end{Figure}
\begin{Figure}
\centering
\begin{tikzpicture}
\draw (0.0,0.0) -- (0.7,0.0) -- (1.3761,0.1812) -- (1.0261,0.7874) -- (1.5211,1.2824) -- (1.1711,0.6761) -- (1.3523,1.3523) -- (2.0523,1.3523) -- (1.8711,2.0284) -- (1.5211,2.6347) -- (1.0261,3.1296) -- (0.6761,2.5234) -- (0.0,2.7046) -- (0.7,2.7046) -- (0.0239,2.5234) -- (-0.3261,3.1296) -- (-0.8211,2.6347) -- (-1.1711,2.0284) -- (-1.3523,1.3523) -- (-0.6523,1.3523) -- (-0.4711,0.6761) -- (-0.8211,1.2824) -- (-0.3261,0.7874) -- (-0.6761,0.1812) -- (0.0,0.0);
\end{tikzpicture}
\captionof*{ }{$\langle n^3 \rangle_{24}$}
\end{Figure}

\end{multicols}

\bibliographystyle{amsplain}

\begin{thebibliography}{10}

\bibitem {A} Weisstein, Eric W. "Walk." From MathWorld--A Wolfram Web Resource.
\bibitem {ds} Golubitsky, M. Introduction to Applied Nonlinear Dynamical Systems and Chaos. New York: Springer-Verlag, 1997.
\bibitem {ifs} Barnsley, M. Fractals Everywhere, 2nd ed. Boston, MA: Academic Press, 1993.
\bibitem {cornu} Weisstein, Eric W. "Cornu Spiral." From MathWorld--A Wolfram Web Resource.
\bibitem {x} Schneider, F. (2017, Aug 25). Functional Walk Generator. From www.walk.fschneider.info

\end{thebibliography}

\end{document}